\documentclass[11pt]{amsart}

\usepackage{amsmath,amscd}
\usepackage{amssymb}
\usepackage{amsthm}
\usepackage{verbatim}
\usepackage[utf8]{inputenc}
\usepackage{mathtools}                          % For \coloneqq needed for \mdef

\newtheorem{Theorem}{Theorem}[section]
\newtheorem*{TheoremNoNumber}{Theorem}
\newtheorem{Lemma}[Theorem]{Lemma}

\newtheorem*{Notational convention}{Notational convention}

\newtheorem{TheoremIntro}{Theorem}

\numberwithin{equation}{section}

\newcommand{\mR}{\mathbb{R}}                    % Formatting for R
\newcommand{\norm}[1]{\lVert #1 \rVert}         % Formatting for the norm
\newcommand{\br}[1]{\langle #1 \rangle}         % Formatting for the inner product

\newcommand{\mdef}{\coloneqq}

\newcommand{\doo}{\partial}
\newcommand{\bigO}{O}

\newcounter{sidenote}
\newcommand{\e}{\text{e}}
\DeclareMathOperator{\Vol}{Vol}

\newcommand{\grad}{\nabla}
\DeclareMathOperator{\Hess}{Hess}
\newcommand{\I}{\mathbb{I}}

\newcommand{\escp}[1]{\mathcal{E}_{#1}}

\newcommand{\classP}[1]{P_{#1}}
\newcommand{\classPd}[1]{P^1_{#1}}

\newcommand{\classE}[1]{E_{#1}}
\newcommand{\classEd}[1]{E^1_{#1}}

\newcommand{\rot}[2]{\e^{i{#2}}{#1}}
\newcommand{\rotn}[2]{\e^{-i{#2}}{#1}}

\newcommand{\ball}[2]{B_{#1}({#2})}
\newcommand{\sphere}[2]{S_{#1}({#2})}

\newcommand{\ho}{o}

\begin{document}

\title[The geodesic ray transform on two-dimensional CH-manifolds]{The geodesic ray transform on two-dimensional
Cartan-Hadamard manifolds} 

\author[J. Lehtonen]{Jere Lehtonen}
\address{Department of Mathematics and Statistics, University of Jyv\"askyl\"a}
\email{jere.ta.lehtonen@jyu.fi}

%\subjclass{53C25, 53C21, 58F17, 35J15}

\date{}
\begin{abstract}
We prove two injectivity theorems for
the geodesic ray transform on
two-dimensional, complete, simply connected Riemannian
manifolds with non-positive Gaussian curvature, also
known as Cartan-Hadamard manifolds. The first
theorem is concerned with bounded non-positive curvature and the second
with decaying non-positive curvature.
\end{abstract}

\maketitle

\section{Introduction and statement of main results}
In \cite{Hel99} Helgason presents the following result:
Suppose that $f$ is  a continuous function in $\mR^2$,
$|f(x)| \leq C(1+|x|)^{-\eta}$ for some $\eta > 2$, and
$Rf = 0$ where $Rf$ is the Radon transform defined by
$$
  Rf(x,\omega) \mdef \int_\mR f(x+t\omega) \, dt
$$
for $x \in \mR^2$ and $\omega \in S^1$.
Then $f = 0$. Since the operator $R$ is linear this corresponds to
the injectivity of the operator.
This result was later improved by Jensen \cite{Jen04}
requiring that $f = \bigO(|x|^{-\eta}), \eta > 1$.

In \cite{Hel94} Helgason presents a similar injectivity result for
the hyperbolic 2-space $H^2$: Suppose $f$ is a continuous function on $H^2$
such that $|f(x)| \leq C \e^{- d_g(x,\ho)}$, where $\ho$ is a fixed point
in $H^2$, and
$$
  \int_{\gamma} f \, ds = 0
$$
for every geodesic $\gamma$ of $H^2$.
Then $f = 0$.

The previous results are concerned with constant curvature spaces. There
are many related results for Radon type transforms on constant curvature
spaces and noncompact homogeneous spaces, see \cite{Hel99},\cite{Hel13}.
These types of spaces possess many symmetries. On the other hand,
there is also a substantial literature related to geodesic ray transforms
on Riemannian manifolds, see e.g. \cite{Muk77}, \cite{Sha94}, \cite{PSU14}.
Here the symmetry assumptions are replaced by curvature or
conjugate points conditions, but the spaces are required to be compact
with boundary.

In this paper we present injectivity results on two-dimensional, complete,
simply connected Riemannian manifolds with non-positive Gaussian
curvature. Such manifolds are called Cartan-Hadamard manifolds, and they
are diffeomorphic to $\mathbb{R}^2$ (hence non-compact) but do not
necessarily have symmetries. In order to prove our results we extend
energy estimate methods used in \cite{PSU13} to the non-compact case.

Suppose $(M,g)$ is such a manifold and we have a continuous function
$f \colon M \to \mR$.
We define the geodesic ray transform $If \colon SM \to \mR$ of the function
$f$ as
$$
  If(x,v) \mdef \int_{-\infty}^\infty f(\gamma_{x,v}(t)) \, dt,
$$
where the unit tangent bundle $SM$ is defined as
$$
  SM \mdef \{(x,v) \in TM \colon |v|_g = 1\}
$$
and $\gamma_{x,v}$ is the unit speed geodesic with
$\gamma_{x,v}(0) = x$ and $\gamma_{x,v}'(0) = v$.
Since we are working on non-compact
manifolds the geodesic ray transform
is not well defined for all continuous functions.
We need to impose decay requirements for the functions
under consideration. Because of the techniques
used we will also impose decay requirements for the first derivatives
of the function.

We denote by $C_0(M)$ the set of functions $f \in C(M)$ such
that for some $p \in M$ one has $f(x) \to 0$ as $d(p,x) \to \infty$.
Suppose $p \in M$ and $\eta \in \mR$.
We define
\begin{align*}
  \classP{\eta}(p,M) &\mdef \{
  f \in C(M) \colon |f(x)| \leq C(1+d_g(x,p))^{-\eta} \text{ for all $x \in M$}\},\\
  \classPd{\eta}(p,M) &\mdef \{
  f \in C^1(M) \colon |\grad f|_g \in \classP{\eta+1}(p,M)\}
  \cap C_0(M).
\end{align*}
and similarly
\begin{align*}
  \classE{\eta}(p,M) &\mdef \{
  f \in C(M) \colon |f(x)| \leq C\e^{-\eta d_g(x,p)} \text{ for all $x \in M$}\},\\
  \classEd{\eta}(p,M) &\mdef \{
  f \in C^1(M) \colon |\grad f|_g \in \classE{\eta}(p,M)\}
  \cap C_0(M).%,\\
\end{align*}

For all $\eta > 0$
we have inclusions
$$
  \classPd{\eta}(p,M) \subset \classP{\eta}(p,M)
$$
and
$$
  \classEd{\eta}(p,M) \subset \classE{\eta}(p,M),
$$
which can be seen by using Lemma \ref{escaping direction}, equation \eqref{geodesic distance}
and the fundamental theorem of calculus. In addition
$$
  \classE{\eta_1}(p,M) \subset \classP{\eta_2}(p,M)
$$
for all $\eta_1,\eta_2 > 0$.

We can now state our first injectivity theorem.
\begin{TheoremIntro}\label{theorem 1}
Suppose $(M,g)$ is a two-dimensional, complete, simply connected 
Riemannian manifold
whose Gaussian curvature satisfies $-K_0 \leq K(x) \leq 0$
for some $K_0$.
Then the geodesic ray transform is injective
on the set $\classEd{\eta}(M) \cap C^2(M)$
for $\eta > \frac{5}{2}\sqrt{K_0}$.
\end{TheoremIntro}

The second theorem considers the case of suitably decaying
Gaussian curvature. By imposing decay requirements for the
Gaussian curvature we are able to relax the decay requirements of
the functions we are considering.

\begin{TheoremIntro}\label{theorem 2}
Suppose $(M,g)$ is a two-dimensional, complete, simply connected 
Riemannian manifold of non-positive
Gaussian curvature $K$ such that
$K \in \classP{\tilde\eta}(p,M)$ for some $\tilde\eta > 2$ and $p \in M$.
Then the geodesic ray transform is injective
on set $\classPd{\eta}(p,M) \cap C^2(M)$ for $\eta > \frac{3}{2}$.
\end{TheoremIntro}

One question arising is of course the existence of manifolds satisfying the
restrictions of the theorems.
By the Cartan-Hadamard theorem such manifolds are always diffeomorphic
with the plane $\mR^2$ so the question is what kind of Gaussian curvatures
we can have on $\mR^2$ endowed with a complete Riemamnian metric?
The following theorem by Kazdan and Warner \cite{KW74} answers this:
\begin{TheoremNoNumber}
Let $K \in C^\infty(\mR^2)$. A necessary and sufficient condition
for there to exist a complete Riemannian metric on $\mR^2$ with
Gaussian curvature $K$ is that
\[
  \lim_{r \to \infty} \inf_{|x| \geq r} K(x) \leq 0.
\]
\end{TheoremNoNumber}

Especially for every non-positive function $K \in C^\infty(\mR^2)$
there exists a metric on $\mR^2$ with Gaussian curvature $K$.

The case where the metric $g$ differs from the euclidean metric
$g_0$ only in some compact set and the Gaussian curvature is
everywhere non-positive is not interesting from the geometric
point of view. By a theorem of Green and Gulliver \cite{GG85} if the metric
$g$ differs from the euclidean metric $g_0$
at most on a compact set and there are no conjugate points, then
the manifold is isometric to $(\mR^2,g_0)$. Since non-positively
curved manifolds can not contain conjugate points this would be the case.

The problem of recovering a function from its integrals
over all lines in the plane goes back to Radon
\cite{Rad17}. He proved the injectivity of the
integral transform nowadays known as the Radon transform
and provided a reconstruction formula.

It is also worth mentioning a counterexample for injectivity
of the Radon transform provided by Zalcman \cite{Zal82}
He showed that on $\mR^2$ there exists a non-zero continuous
function which is $\bigO(|x|^{-2})$ along every line and integrates to zero over any line.
See also \cite{AG93},\cite{Arm94}.

This work is organized as follows.
In the second section we describe the geometrical
setting of this work and present some results mostly concerning
behaviour of geodesics. 
The third section is about the geodesic ray transform.
In the fourth section we derive estimates for the growth
of Jacobi fields in our setting and use those
to prove useful decay estimates.
The fifth section contains the proofs of our main theorems.
%As mentioned, those rely on certain type of Pestov identity, which is presented there too.

\begin{Notational convention}
Throughout this work
we denote by $C(a,b,\dots)$ (with a possible
subscript) a constant
depending on $a,b,\dots$ The value of the constant
may vary from line to line.
\end{Notational convention}

\subsection*{Acknowledgement}
This work is part of the PhD research of the author.
The author is partly supported by the Academy of Finland.
The author wishes to thank professor M.~Salo
for many helpful ideas and discussions regarding this work.
The author is also thankful for J.~Ilmavirta
for many insightful comments.

\section{The setting of this work and preliminaries}
Throughout this paper we assume $(M,g)$ to be a two-dimensional, complete,
simply connected manifold with non-positive Gaussian curvature
$K$. By the Cartan-Hadamard theorem the exponential map $\exp_x \colon
T_x M \to M$ is a diffeomorphism for every point $x \in M$ . Thereby we have global normal coordinates
centered at any point and we could equivalently work with $(\mR^2,\tilde g)$ where $\tilde g$
is pullback of the metric $g$ by exponential map, but we choose to present this work in
the general setting of $(M,g)$.

We make the standing assumption of unit-speed parametrization for
geodesics.
If $x \in M$ and $v \in T_xM$ is such that $|v|_g = 1$ we denote
by $\gamma_{x,v} \colon \mR \to M$ the geodesic with $\gamma_{x,v}(0) = x$
and $\gamma_{x,v}'(0) = v$.

The fact that for every point the exponential map
is a diffeomorphism implies that every pair of distinct points can be
joined by an unique geodesic. Furthermore, by using the triangle inequality, we have
\begin{equation}\label{geodesic distance}
  d_g(\gamma_{x,v}(t),p) \geq d_g(\gamma_{x,v}(t),x) - d_g(x,p)
  = |t| - d_g(x,p)
\end{equation}
for every $p \in M$ and $(x,v) \in SM$.

Because of the everywhere non-positive Gaussian curvature, the function
$t \mapsto d_g(\gamma(t),p)$ is convex on $\mR$ and the function
$t \mapsto d_g(\gamma(t),p)^2$ is strictly convex
on $\mR$ for every geodesic $\gamma$ and point $p \in M$ (see e.g.~\cite{Pet98}).

We say that the geodesic $\gamma_{x,v}$ is escaping with respect to point $p$
if function $t \mapsto d_g(\gamma_{x,v}(t),p)$ is strictly increasing
on the interval $[0,\infty)$. The set of such geodesics is denoted
by $\escp{p}(M)$.

\begin{Lemma}\label{escaping direction}
Let $p \in M$ and $(x,v) \in SM$. At least one of geodesics
$\gamma_{x,v}$ and $\gamma_{x,-v}$ is in set $\escp{p}(M)$.
\end{Lemma}

\begin{proof}
The function $t \mapsto d_g(\gamma(t),p)^2$ is strictly convex
on $\mR$ so it has a strict global minimum. Therefore the
function $t \mapsto d_g(\gamma(t),p)$ also has a strict
global minimum, which implies that at least one of functions
$t \mapsto d_g(\gamma_{x,v}(t),p)$
and $t \mapsto d_g(\gamma_{x,-v}(t),p)$ is strictly increasing
on the interval $[0,\infty)$.
\end{proof}

If the geodesic $\gamma_{x,v}$ belongs to $\escp{p}(M)$
equation \eqref{geodesic distance} implies the estimate
\begin{equation}\label{geodesic distance estimate}
  d_g(\gamma_{x,v}(t),p) \geq
  \begin{cases}
    d_g(x,p),&\text{if }0\leq t \leq 2d_g(x,p),\\
    t - d_g(x,p),&\text{if }2d_g(x,p) < t.
  \end{cases}
\end{equation}

The manifold $M$ is two-dimensional and oriented
and so is also the tangent space $T_xM$ for every $x\in M$.
Thus given $v \in T_xM$ we can define $\rot{v}{t} \in T_xM, t \in \mR$,
to be the unit vector obtained by rotating the vector $v$ by an angle $t$.
We will use the shorthand notation $v_\perp \mdef \rotn{v}{\pi/2}$.

The unit tangent bundle $SM$ is a 3-dimensional manifold
and there is a natural Riemannian metric on it, namely
the Sasaki metric \cite{Pat99}. The volume form given by this metric
is denoted by $d\Sigma^3$.

On the manifold $SM$ we have the geodesic flow
$\varphi_t \colon SM \to SM$ defined by
$$
  \varphi_t(x,v) = (\gamma_{x,v}(t),\gamma_{x,v}'(t)).
$$
We denote by $X$ the
vector field associated with this flow.
We define flows $p_t,h_t \colon SM \to SM$ as
\begin{align*}
  &p_t(x,v) \mdef (x,\rot{v}{t}),\\
  &h_t(x,v) \mdef (\gamma_{x,v_\perp}(t),Z(t)),
\end{align*}
where $Z(t)$ is the parallel transport of the vector $v$
along the geodesic $\gamma_{x,v_\perp}$,
and denote the associated vector fields by $V$ and $X_\perp$.

These three vector fields form a global orthonormal frame for $T(SM)$
and we have following structural equations (see \cite{PSU13})
\begin{align*}
  [X,V] &= X_\perp,\\
  [V,X_\perp] &= X, \\
  [X,X_\perp] &= -KV,
\end{align*}
where $K$ is the Gaussian curvature of the manifold $M$.

Let $f \colon U \subset M \to \mR$ be
such that $|\grad f|_g = 1$. Then level sets
of the function $f$ are submanifolds of $M$.
The second fundamental form $\I$ on such
a level set is defined as 
$$
   \I(v,w) \mdef \Hess(f)(v,w),
$$
where $v,w \perp \grad f$ and $\Hess(f)$ is the covariant
Hessian (see \cite{Pet98}).

Suppose that $p \in M$.
Denote by $\ball{p}{r}$ the open geodesic ball
with radius $r$, and by $\sphere{p}{r}$ its boundary.

\begin{Lemma}\label{spheres strictly convex}
For every $p \in M$ and $r> 0$ the geodesic
ball $\ball{p}{r}$ has a strictly convex boundary,
i.e.~the second fundamental form of $\sphere{p}{r}$
is positive definite.
\end{Lemma}
\begin{proof}
Suppose $x \in \sphere{p}{r}$ and $v$ is tangent to
$\sphere{p}{r}$ at $x$. Denote $f(y) = d_p(y,p)$.
We have
$$
  \Hess(f^2)(x) = 2f(x)\Hess f(x) + 2\,d_xf \otimes d_xf
$$
and thus
$$
  \Hess(f^2)(x)(v,v) = 2f(x)\Hess f(x)(v,v)
$$
since $d_xf(v) = \br{\grad f(x), v}_g = 0$.

Since the function $t \to d(\gamma_{x,v}(t),p)^2$ is strictly convex
we get
$$
  \Hess(f^2)(x)(v,v) = \frac{d^2}{dt^2}((f^2 \circ \gamma_{x,v})(t))\Big|_{t=0} > 0.
$$
Therefore $\Hess(f^2)$ is positive definite in tangential directions
and so is also $\Hess f$
\end{proof}

Equivalently, the boundary of $\ball{p}{r}$ is strictly convex if and only if
every geodesics starting from a boundary point in a direction
tangent to boundary stays outside $\ball{p}{r}$ for small positive
and negative times and has a second order contact at time $t=0$. From this we see that
if $x \in M$ and $v$ is tangent to $\sphere{p}{d_g(x,p)}$ then
function $t \mapsto d_g(\gamma_{x,v}(t),p)^2$ has a global minimum at 
$t = 0$.

\begin{Lemma}\label{vertical escaping}
Suppose $p \in M$ and $(x,v) \in SM$ is such that
$\gamma_{x,v} \in \escp{p}(M)$ and $v$ is not tangent
to $\sphere{p}{d(x,p)}$. Then $\gamma_{p_s(x,v)} \in \escp{p}(M)$ for
small $s$.

If $v$ is tangent then $\gamma_{p_t(x,v)} \in \escp{p}(M)$ for
either small $t > 0$ or small $t < 0$.
\end{Lemma}
\begin{proof}
Suppose first that $v$ is not tangent to
$\sphere{p}{d(x,p)}$. Then it must be that
$$
  \frac{d}{dt} d_g(\gamma_{x,v}(t),p)^2\Big|_{t=0} > 0.
$$
The function $s \mapsto \frac{d}{dt} d(\gamma_{p_s({x,v})}(t),p)^2$
is continuous and hence
\begin{equation}\label{d2gstdgtz}
  \frac{d}{dt} d_g(\gamma_{p_s({x,v})}(t),p)^2\Big|_{t=0} > 0
\end{equation}
for small $s$. Thus $\gamma_{p_s(x,v)} \in \escp{p}(M)$.

If $v$ is tangent to $\sphere{p}{d_g(x,p)}$ then
$$
  \frac{d}{dt} d_g(\gamma_{x,v}(t),p)^2\Big|_{t=0} = 0
$$
and \eqref{d2gstdgtz} holds either for small positive
$s$ or for small negative $s$.
\end{proof}

\begin{Lemma}\label{horisontal escaping}
Suppose $p \in M$ and $(x,v) \in SM$ is such that
$\gamma_{x,v} \in \escp{p}(M)$. Then $\gamma_{h_s(x,v)} \in \escp{p}(M)$ for
small $s$.
\end{Lemma}
\begin{proof}
If $v$ is not tangent to $\sphere{p}{d_g(x,p)}$ then
proof is as for the flow $p_s$. If $v$ is tangent to
$\sphere{p}{d_g(x,p)}$ then $\gamma_{h_s(x,v)}(0)$
is tangent to $\sphere{p}{d_g(x,p)+ s}$
or $\sphere{p}{d_g(x,p) - s}$ and thus
$\gamma_{h_s(x,v)} \in \escp{p}(M)$.
\end{proof}

The next lemma is equation \eqref{geodesic distance estimate}
for $\gamma_{h_s}$ and $\gamma_{p_s}$.
\begin{Lemma}\label{geodesic distance estimate for flows}
For all $s$ such that $\gamma_{h_s(x,v)} \in \escp{p}(M)$
we have
$$
d_g(\gamma_{h_s(x,v)}(t),p)
  \geq \begin{cases}
     d_g(x,p) - s,&  0 \leq t \leq 2d_g(x,p), \\
     t - d_g(x,p) - s, & t > 2d_g(x,p).
\end{cases}
$$
For all $s$ such that $\gamma_{p_s(x,v)} \in \escp{p}(M)$
we have
$$
d_g(\gamma_{p_s(x,v)}(t),p)
  \geq \begin{cases}
     d_g(x,p),&  0 \leq t \leq 2d_g(x,p), \\
     t - d_g(x,p), & t > 2d_g(x,p).
\end{cases}
$$
\end{Lemma}
\begin{proof}
We have for $\gamma_{h_s(x,v)}$
by triangle inequality
$$
  d_g(\gamma_{h_s(x,v)}(0),p)
  \leq d_g(\gamma_{h_s(x,v)}(0),x) + d_g(x,p)
  = s + d_g(x,p)
$$
and furthermore
\begin{align*}
  d_g(\gamma_{h_s(x,v)}(t),\gamma_{h_s(x,v)}(0))
  &\leq
  d_g(\gamma_{h_s(x,v)}(t),p) + d_g(\gamma_{h_s(x,v)}(0),p) \\
  &\leq
  d_g(\gamma_{h_s(x,v)}(t),p) + s + d_g(x,p).
\end{align*}
so
$$
  t-s - d_g(x,p) \leq d_g(\gamma_{h_s(x,v)}(t),p).
$$

By triangle inequality
$$
  d_g(x,p) \leq d_g(\gamma_{h_s(x,v)}(0),p) + d_g(\gamma_{h_s(x,v)}(0),x)
  = d_g(\gamma_{h_s(x,v)}(0),p) + s.
$$
Because $\gamma_{h_s(x,v)}$ is in $\escp{p}(M)$ we get for $t \geq 0$
$$
  d_g(\gamma_{h_s(x,v)}(t),p)
  \geq d_g(\gamma_{h_s(x,v)}(0),p)
  \geq d_g(x,p) - s.
$$

The result for $\gamma_{h_s(x,v)}$ follows by combining these estimates.
For $\gamma_{p_s(x,v)}$ proof is similar, but we have
$d_g(\gamma_{p_s(x,v)}(0),x) = 0$.
\end{proof}

\section{The geodesic ray transform}

As mentioned in the introduction the geodesic ray transform  $If \colon SM \to \mR$
of a function $f \colon SM \to \mR$ is defined by
$$
  If(x,v) \mdef \int_{-\infty}^\infty f(\gamma_{x,v}(t)) \, dt.
$$

\begin{Lemma}
The geodesic ray transform is well defined for
$f \in \classP{\eta}(p,M)$ for $\eta > 1$.
%and for $f \in \classE{\eta}(p,M), \eta > 0$.
\end{Lemma}

\begin{proof}
Let $(x,v) \in SM$. Since $If(\gamma_{x,v}(t),\gamma_{x,v}'(t)) = If(x,v)$
for all $t \in \mR$, we can assume $x$ to be such that 
$$
  \min_{t \in\mR} d_g(\gamma_{x,v}(t),p) = d_g(x,p).
$$
Such a point always exists on any geodesic $\gamma$ since
the mapping $t \mapsto d_g(\gamma(t),p)^2$ is strictly convex.

By \eqref{geodesic distance} we then have
$$
  d_g(\gamma_{x,v}(t),p)
  \geq
  \begin{cases}
    d_g(x,p),&\text{if } |t| \leq 2d_g(x,p),\\
    |t| - d_g(x,p),&\text{if }2d_g(x,p) < |t|.
  \end{cases}
$$
Hence for $f \in \classP{\eta}(p,M), \eta > 1,$
\begin{align*}
  |If(x,v)|
  &\leq \int_{-\infty}^\infty |f(\gamma_{x,v}(t))| \, dt
  \leq \int_{-\infty}^\infty \frac{C}{(1+d_g(\gamma_{x,v}(t),p))^\eta} \, dt \\
  &\leq C \left( \int_0^{2d_g(x,p)} \frac{1}{(1 + d_g(x,p))^\eta} \, dt
  + \int_{2d_g(x,p)}^{\infty} \frac{1}{(1 + t - d_g(x,p))^\eta} \, dt \right) \\
  &\leq C \left( \frac{2d_g(x,p)}{(1 + d_g(x,p))^\eta}
  + \frac{1}{(\eta-1)(1+d_g(x,p))^{\eta-1}}\right) \\
  &\leq  \frac{C(\eta)}{(1 + d_g(x,p))^{\eta-1}}.\qedhere
\end{align*}
\end{proof}

Given a function $f$ on $M$ we define the
function $u^f \colon SM \to \mR$ by
$$
  u^f(x,v) = \int_0^\infty f(\gamma_{x,v}(t)) \, dt.
$$
We observe that
$$
  If(x,v) = u^f(x,v) + u^f(x,-v)
$$
for all $(x,v) \in SM$ whenever all the functions
are well defined.

%We need to know about decay properties of the function $u^f$.
In the next lemma we assume that $f$ is such that $If \equiv 0$
since those functions are in our interest.

\begin{Lemma}
Suppose $p \in M$ and $f$ is a function on $M$ such that $If \equiv 0$.
\begin{enumerate}
  \item If $f \in \classE{\eta}(p,M)$ for some $\eta > 0,$ then
    $$ |u^f(x,v)| \leq C(\eta) (1+d_g(x,p))\e^{-\eta d_g(x,p)}.$$
  \item If $f \in \classP{\eta}(p,M)$ for some $\eta > 1,$ then
    $$ |u^f(x,v)| \leq \frac{C(\eta)}{(1+d(x,p))^{\eta-1}}.$$
\end{enumerate}
\end{Lemma}

\begin{proof}
Since $If(x,v) = 0$ we have $|u^f(x,v)| = |u^f(x,-v)|$ for
all $(x,v) \in SM$. Thus, by Lemma \ref{escaping direction},
we can assume $(x,v)$ to be such that $\gamma_{x,v} \in \escp{p}(M)$.

If $f \in \classP{\eta}(p,M), \eta > 1$,
using the estimate \eqref{geodesic distance estimate} we obtain
\begin{align*}
  |u^f(x,v)|
  &\leq C\left( \int_0^{2d_g(x,p)} \frac{1}{(1+ d_g(\gamma_{x,v}(t),p))^\eta} \, dt\right.\\
  &+ \left. \int_{2d_g(x,p)}^\infty \frac{1}{(1+d_g(\gamma_{x,v}(t),p))^\eta} \, dt\right)\\
  &\leq \frac{C(\eta)}{(1+d_g(x,p))^{\eta-1}}.
\end{align*}

Similarly for $f \in \classE{\eta}(p,M), \eta > 0$, we get
\begin{align*}
  |u^f(x,v)|
  &\leq C\left( \int_0^{2d_g(x,p)} \e^{-\eta d_g(x,p)} \, dt
  + \int_{2d_g(x,p)}^\infty \e^{-\eta (t - d_g(x,p))} \, dt\right)\\
  &\leq C(\eta)  (1+d_g(x,p))\e^{-\eta d_g(x,p)}.\qedhere
\end{align*}
\end{proof}

Next we prove that $Xu^f = -f$, which can be seen
as a reduction to transport equation.
This idea is explained in details in \cite{PSU13}.

\begin{Lemma}\label{Xu = -f}
Suppose $f \in \classPd{\eta}(p,M)$ for some $\eta > 1$
and $If = 0$. Then $Xu^f(x,v) = -f(x)$ for every $(x,v) \in SM$.
\end{Lemma}
\begin{proof}
We begin by observing that
$$
  X(If(x,v)) = Xu^f(x,v) + X (u^f(x,-v)) = 0 
$$
so $Xu^f(x,v) = -X(u^f(x,-v))$. Hence we can assume
the geodesic $\gamma_{x,v}$ to be in $\escp{p}(M)$ by
Lemma \ref{escaping direction}.

We have
\begin{align*}
  Xu^f(x,v)
  &= \frac{d}{ds} u^f(\varphi_s(x,v))\Big|_{s=0}
  = \frac{d}{ds} \int_0^\infty f(\gamma_{\varphi_s(x,v)}(t)) \, dt\Big|_{s=0} \\
  &= \int_0^\infty \frac{d}{ds} f(\gamma_{x,v}(s+t))\Big|_{s=0} \, dt
\end{align*}
where the last step needs to be justified.

Since we assumed our geodesic to be in $\escp{p}(M)$,
for $t,s \geq 0$ it holds
\begin{align*}
  |\frac{d}{ds} f(\gamma_{x,v}(t+s))|
  &= |d_{\gamma_{x,v}(t+s)}f(\gamma'_{x,v}(t+s))| \\
  &\leq \frac{C}{(1+d_g(\gamma_{x,v}(t+s),p))^{\eta+1}} \\
  &\leq \frac{C}{(1+d_g(\gamma_{x,v}(t),p))^{\eta+1}}.
\end{align*}
Using estimate \eqref{geodesic distance} as in the earlier proofs we obtain
\begin{align*}
  \int_0^\infty |\frac{d}{ds} f(\gamma_{x,v}(s+t))| \, dt
  &\leq \int_0^\infty \frac{C}{(1+d_g(\gamma_{x,v}(t),p))^{\eta+1}}\, dt \\
  &\leq \frac{C(\eta)}{(1+d_g(x,p))^\eta},
\end{align*}
which shows that the last step earlier is justified by the
dominated convergence theorem. 

Since
$$
  \frac{d}{ds} f(\gamma_{x,v}(t+s))\Big|_{s=0}
  = \frac{d}{dt} f(\gamma_{x,v}(t))
$$
and $f(\gamma_{x,v}(t)) \to 0$ as $t \to \infty$ we have
$$
  \int_0^\infty \frac{d}{ds} f(\gamma_{x,v}(s+t))\Big|_{s=0} \, dt
  = -f(x)
$$
by the fundamental theorem of calculus.
\end{proof}

\section{Regularity and decay of $u^f$}
In order to prove our main theorems we need to prove
$C^1$-regularity for $u^f$ given that the function $f$
has suitable regularity and decay properties.
For that we derive estimates for functions
$X_\perp u^f$ and $Vu^f$.
To prove the estimates for functions $X_\perp u^f$ and $Vu^f$ we
will proceed as in the case of $Xu = -f$ (Lemma \ref{Xu = -f}).
In the proof we calculated
$$
  \frac{d}{ds} f(\gamma_{\varphi_s(x,v)}(t))\Big|_{s=0} =
  d_{\gamma_{\varphi_s(x,v)}(t)}f(\frac{d}{ds}\gamma_{\varphi_s(x,v)}(t)\Big|_{s=0}).
$$

We can interpret $\frac{d}{ds}\gamma_{\varphi_s(x,v)}(t)\big|_{s=0}$
as a Jacobi field along the geodesic $\gamma_{x,v}$ since it is just the
tangent vector field.
For $X_\perp u^f$ and $Vu^f$ we proceed in a similar manner, the difference
being that the geodesic flow $\varphi_t$ is replaced with the flows $h_t$ and $p_t$
respectively.

Given geodesic $\gamma_{x,v}$ we denote
$$
  J_{\gamma_{x,v},h}(s,t) = \frac{d}{dr} \gamma_{h_r(x,v)}(t)\Big|_{r=s}
$$
and $J_{\gamma_{x,v},p}$ similarly.
Then $J_{\gamma_{x,v},h}(s,t)$ is a Jacobi field along geodesic $\gamma_{h_s(x,v)}$  for fixed $s$.
We will write $J_h(s,t)$ when it is clear from the context what
the undelying geodesic is.
We will also use shorthand notation $J_h(t) = J_h(0,t)$ and
$J_p(t) = J_p(0,t)$.

The Jacobi fields obtained in this manner
turn out to be normal with initial data
(see \cite{PU04})
\begin{align*}
J_h(s,0) = 1,\quad D_t J_h(s,0) = 0,\\
J_p(s,0) = 0,\quad D_t J_p(s,0) = 1.
\end{align*}

We need to have estimates for the growth of these two Jacobi
fields in particular. The first lemma giving estimates for the growth
is based on comparison theorems for Jacobi fields.
See for example \cite[Theorem 4.5.2]{Jos08}.
\begin{Lemma}\label{jacobi growth exponential}
Suppose $|K(x)| \leq K_0$ and $\gamma$ is a geodesic.
Then for Jacobi fields $J_p$ and $J_h$
along a geodesic $\gamma$ it holds that
\begin{align*}
  &|J_p(t)| \leq C(K_0)\e^{\sqrt{K_0}t},\\
  &|J_h(t)| \leq C(K_0)\e^{\sqrt{K_0}t},
\end{align*}
for $t \geq 0$.
\end{Lemma}

This lemma tells us that these Jacobi fields
will grow at most exponentially in presence of bounded curvature.
If the curvature happens to decay suitably we will see that these Jacobi fields will
grow only at a polynomial rate.

If $J(t)$ is a normal Jacobi field along a geodesic
$\gamma$ then we can write $J(t) = u(t)E(t)$ where
$u$ is a real valued function and $E(t)$ is a unit
normal vector field along $\gamma$. From the
Jacobi equation it follows that $u$ is a solution to
$$
  u''(t) + K(\gamma(t))u(t) = 0
$$
for $t \geq 0$ with initial values $u(0) = \pm|J(0)|$
and $u'(0) = \pm|D_t J(0)|$.

This leads us to consider an ordinary differential equation
\begin{equation}\label{ivp jacobi scalar non-hg}
  \begin{cases}
   &u''(t) + K(t)u(t) = 0,\quad t \geq 0,\\
   &u(0) = c_1,\\
   &u'(0) = c_2,
  \end{cases}
\end{equation}
for continuous $K$, where $c_1,c_2 \in \mR$. Note that for
$J_h$ and $J_p$ the constants $c_1$ and $c_2$ are either $0$ or $\pm 1$.

Waltman \cite{Wal64} proved that if $u$ is a solution
to \eqref{ivp jacobi scalar non-hg} with $K$ such that
$$
  \int_0^\infty t |K(t)| \, ds < \infty
$$
then $\lim_{t \to \infty} u(t)/t$ exists. We reproduce essential parts of the proof
in order to obtain a more quantitative estimate for the growth
of the solution $u$.
\begin{Lemma}\label{jacobi growth polynomial waltman}
Suppose $u$ is a solution to \eqref{ivp jacobi scalar non-hg} with
$$
  M_K \mdef \int_0^\infty s|K(s)| \, ds < \infty.
$$
and $c_1 = 1, c_2 = 0$ or other way around.
Then
$$
  |u(t)| \leq C_1 t + C_2
$$
for all $t \geq 0$ where $C_1,C_2 \geq 0$.
\end{Lemma}
\begin{proof}
We define
$A(t) = u'(t)$ and $B(t) = u(t) - tu'(t)$ so $u(t) = A(t)t + B(t)$.
Fix $t_0 > 0$. For all $t > t_0$ it holds
\begin{align*}
  &A(t) = A(t_0) - \int_{t_0}^t K(s)s\left(A(s) + \frac{B(s)}{s}\right) \, ds,\\
  &B(t) = B(t_0) + \int_{t_0}^t K(s)s^2\left(A(s) + \frac{B(s)}{s}\right) \, ds.
\end{align*}
If we define
$|v(t)| = |A(t)| + |B(t)/t|$
we have 
$$
  |v(t)| \leq |v(t_0)| + 2 \int_{t_0}^t s|K(s)||v(s)| \, ds.
$$
By a theorem of Viswanatham \cite{Vis63}
it holds $|v(t)| \leq \psi(t)$ on $[t_0,\infty)$ where $\psi$ is
a solution to
$$
  \psi'(t) = 2t|K(t)|\psi(t)
$$
with $\psi(t_0) = |v(t_0)|$. Hence
$$
  \psi(t) = |v(t_0)| \e^{2\int_{t_0}^t s|K(s)| \, ds} \leq  |v(t_0)|\e^{2 M_K}
$$
and furthermore
$$
  |u(t)| = |t v(t)|
  \leq t \e^{2 M_K}|v(t_0)|
$$
for $t \geq t_0$.

Then we need to estimate $|v(t_0)|$. In order to do so
we need estimates for $|u(t_0)|$ and $|u'(t_0)|$.
We can apply Lemma \ref{jacobi growth exponential} to get
$$
  |u(t)| \leq C(K_0)\e^{\sqrt{K_0}t_0}
$$
on interval $[0,t_0]$ where we have denoted $K_0 = \sup_{t \in [0,t_0]} |K(t)|$.
By integrating equation \eqref{ivp jacobi scalar non-hg}
we obtain
\begin{align*}
  |u'(t_0)|
  &\leq |u'(0)| + \int_0^{t_0} |K(t)| |u(t)|\, ds \\
  &\leq |c_2| + |\sup_{t \in [0,t_0]}u(t)|K_0 t_0
\end{align*}
Thus
\begin{align*}
  |v(t_0)|
  &\leq |A(t_0)| + |B(t_0)/t_0| \leq |u(t_0)/t_0| + 2 |u'(t_0)|\\
  &\leq C(K_0)(\frac{1}{t_0} +  2K_0t_0)\e^{\sqrt{K_0}t_0}
   + 2.
\end{align*}

By combining the estimates for intervals $[0,t_0]$ and
$[t_0,\infty)$ and setting $t_0 = 1$ we obtain that 
$$
  |u(t)| \leq t \e^{2 M_K}|v(1)| + C(K_0)\e^{\sqrt{K_0}}
$$
for $t \geq 0$.
\end{proof}

\begin{Lemma}\label{jaboci on set G}
Suppose $|K(x)| \leq K_0$ and
that $G$ is a set of geodesics such that
$$
  M_G \mdef \sup_{\gamma \in G} \int_0^\infty t|K(t)| \, dt < \infty.
$$
Let $\gamma \in G$.
Then for Jacobi fields $J_p$ and $J_h$ along geodesic $\gamma$ holds
\begin{align*}
  &|J_p(t)| \leq C(M_G)t,\\
  &|J_h(t)| \leq C(M_G)(t+1).
\end{align*}
for all $t \geq 0$. Especially the constants do not depend
on the geodesic $\gamma$.
\end{Lemma}
\begin{proof}
Suppose geodesic $\gamma_{x,v}$ is in $G$.
By Lemma \ref{jacobi growth polynomial waltman} we obtain
\begin{align*}
  &|J_h(t)| \leq C_1 t + C_2,\\
  &|J_p(t)| \leq C_1 t + C_2.
\end{align*}
From the proof of that lemma we see that
constants $C_1$ and $C_2$ above depend on the lower bound for $K$ and the quantity
$$
  \int_0^\infty -tK(\gamma_{x,v}(t)) \, dt.
$$
Since this quantity is bounded from above by $M_G$
we can estimate constants $C_1$ and $C_2$ by above and get rid of the dependence
on the geodesic $\gamma_{x,v}$. So the constants depend only on the Gaussian curvature
$K$ and the initial conditions.

Furthermore, since $|J_p(0)| = 0$ we can drop the constant
$C_2$ in the estimate for $J_p(t)$ by making $C_1$ accordingly
larger.
\end{proof}

Next lemma is a straightforward corollary of the preceding lemma.

\begin{Lemma}\label{jacobi growth polynomial}
Suppose $K \in \classP{\eta}(p,M)$ for some $\eta > 2$.
If $\gamma \in \escp{p}(M)$ then for Jacobi fields $J_p$ and $J_h$
along geodesic $\gamma$ one has
\begin{align*}
  &|J_p(t)| \leq Ct,\\
  &|J_h(t)| \leq C(t+1),
\end{align*}
for all $t \geq 0$, where the constants
do not depend on the geodesic $\gamma$.
\end{Lemma}
\begin{proof}
Since $K \in \classP{\eta}(p,M),
\eta > 2$, we have
\begin{equation*}\label{geodesic weighted curvature sup}
  \sup_{\gamma \in \escp{p}(M)} \int_0^\infty -K(\gamma(t)) t \, dt < \infty.\qedhere
\end{equation*}
\end{proof}

With Lemmas \ref{jacobi growth exponential} and \ref{jacobi growth polynomial}
we can derive estimates for $X_\perp u^f$ and $Vu^f$.

\begin{Lemma}\label{derivative estimates}
Let $f \in C(M)$ be such that $If = 0$.
\begin{enumerate}
\item If $|K(x)| \leq K_0$ and $f \in \classEd{\eta}(p,M)$ for some $\eta > \sqrt{K_0}$,
  then
  $$
    |X_\perp u^f(x,v)| \leq C(\eta,K_0)\e^{(2\sqrt{K_0} - \eta)d_g(x,p)}
  $$
  for all $(x,v) \in SM$.
\item If $f \in \classPd{\eta}(p,M)$ for some $\eta > 1$
  and $K \in \classP{\tilde\eta}(p,M)$ for some $\tilde\eta > 2$, then
  $$
    |X_\perp u^f(x,v)| \leq \frac{C(\eta)}{(1+d_g(x,p))^{\eta-1}}
  $$
  for all $(x,v) \in SM$.
\end{enumerate}

Both estimates hold also if $X_\perp$ is replaced by $V$.
\end{Lemma}

\begin{proof}
Let us first notice that since $If = 0$, it holds
$|X_\perp u^f(x,-v)| = |X_\perp u^f(x,v)|$ for all $(x,v) \in SM$.
Thus we will assume that $v$ is such that
$\gamma_{x,v} \in \escp{p}(M)$.

Firts we note that
$$
  \frac{d}{ds} f(\gamma_{h_s(x,v)}(t))
  = d_{\gamma_{h_s(x,v)}(t)}f(J_h (s,t)).
$$
By definition
\begin{align*}
  X_\perp u^f(x,v)
  &= \frac{d}{ds} \int_0^\infty f(\gamma_{h_s(x,v)}(t)) \, dt \Big|_{s=0}
  = \int_0^\infty \frac{d}{ds} f(\gamma_{h_s(x,v)}(t)) \, dt \Big|_{s=0}\\
  &= \int_0^\infty d_{\gamma_{x,v}(t)} f(J_h(s,t)) \, dt
\end{align*}
where the second equality holds by the dominated convergence theorem provided that
there exists function $F \in L^1([0,\infty))$ such that
\begin{equation}\label{dominating function}
  \left|\frac{d}{ds} f(\gamma_{h_s(x,v)}(t))\right| \leq F(t)
\end{equation}
for all $t \geq 0$ and for small non-negative $s$.

Lemma \ref{horisontal escaping} states that for small $s$ it holds
that $\gamma_{h_s(x,v)} \in \escp{p}(M)$.
Hence in the first case
using Lemmas \ref{geodesic distance estimate for flows}
and \ref{jacobi growth exponential}
we get
\begin{align*}
  |d_{\gamma_{h_s(x,v)}(t)} f(J_h(s,t))|
  &\leq C(K_0)\e^{\sqrt{K_0}t} \e^{-\eta d_g(\gamma_{h_s(x,v)}(t),p)} \\
  &\leq
  \begin{cases}
     C(K_0)\e^{\eta s}\e^{\sqrt{K_0}t}\e^{-\eta d_g(x,p)},&  0 \leq t \leq 2d_g(x,p), \\
     C(K_0)\e^{\eta s}\e^{\sqrt{K_0}t}\e^{-\eta (t-d_g(x,p))}, & t > 2d_g(x,p),
  \end{cases}
\end{align*}
and thus
$$
  \int_0^\infty |d_{\gamma_{h_s(x,v)}(t)} f(J_h(s,t))|
  \leq C(\eta,K_0)\e^{\eta s}\e^{(2\sqrt{K_0} - \eta)d_g(x,p)}.
$$

In the second case we obtain
$$
  |d_{\gamma_{h_s(x,v)}(t)}f(J_h(s,t))|
  \leq
  \begin{cases}
     \frac{C(t+1)}{(1 + d_g(x,p) - s)^{\eta+1}},&  0 \leq t \leq 2d_g(x,p), \\
     \frac{C(t+1)}{(1 + t -d_g(x,p) - s)^{\eta+1}}, & t > 2d_g(x,p).
  \end{cases}
$$
Therefore
$$
  \int_0^\infty |d_{\gamma_{h_s(x,v)}(t)} f(J_h(s,t))| \, dt
  \leq
  \frac{C(\eta)}{(1-s+d_g(x,p))^{\eta-1}}.
$$
From these estimates we see that
such a function $F$ exists in both cases. Setting
$s=0$ gives the estimates for $|X_\perp u^f(x,v)|$.

In case of $V$ instead of $X_\perp$ we
proceed in the same manner. First we notice that $|Vu^f(x,-v)| = |Vu^f(x,v)|$
for all $(x,v) \in SM$. Thus we will assume that $v$ is such that
$\gamma_{x,v} \in \escp{p}(M)$. In addition we will assume $v$ to be such that
$\gamma_{p_s(x,v)} \in \escp{p}(M)$ for small
non-negative $s$, this can be done by Lemma \ref{vertical escaping}.
The rest of the proof is then similar.
\end{proof}

From this result we see that if $f$ is a $C^1$-function
with suitable decay properties then $u^f$ is in
$C^1(SM)$. Later we will approximate $u^f$ with
functions $u^{f_k} \in C^2(SM)$ where functions
$f_k$ are compactly supported $C^2$-functions on $M$.
The following lemma shows that functions $u^{f_k}$ are indeed in $C^2(SM)$.

\begin{Lemma}\label{C^2 for compactly supported}
Suppose that $f \in C^2(M)$ is compactly supported. Then
$u^f \in C^2(SM)$.
\end{Lemma}
\begin{proof}
Since $f$ is compactly supported we have
\begin{align*}
  &Xu^f(x,v) = -f(x),\\
  &X_\perp u^f(x,v) = \int_0^\infty d_{\gamma_{x,v}(t)}f(J_h(t)) \, dt,\\
  &V u^f(x,v) = \int_0^\infty d_{\gamma_{x,v}(t)}f(J_p(t)) \, dt.
\end{align*}

From the structural equations and the knowledge that
$Xu^f = -f$ we can deduce that
$VXu^f, XVu^f, X_\perp Xu^f, X X_\perp u^f$ and
$X^2u^f$ exist.

With other means we have to check that $V^2  u^f, X_\perp^2 u^f$ and
$VX_\perp u^f$ (or equivalently $X_\perp V u^f$) exist.

Let us calculate a formula for $VX_\perp u^f(x,v)$
and from that we see the existence.
By definition
\begin{align*}
  V X_\perp u^f(x,v)
  &= \frac{d}{ds} X_\perp u^f(p_s(x,v))\Big|_{s=0} \\
  &= \frac{d}{ds} \int_0^\infty d_{\gamma_{p_s(x,v)}(t)}f(J_{\gamma_{p_s(x,v)},h}(t))  \, dt\Big|_{s=0}.
\end{align*}
We write
$$
  d_{\gamma_{p_s(x,v)}(t)}f(J_{\gamma_{p_s(x,v)},h}(t))
  = \br{\grad f(\gamma_{p_s(x,v)}(t)),J_{\gamma_{p_s(x,v)},h}(t)}.
$$
Since
$$
  \br{ D_s \grad f(\gamma_{p_s(x,v)}(t)), J_{\gamma_{p_s(x,v)},h}(t)}
  = \Hess f (\gamma_{p_s(x,v)})(J_p(s,t),J_{\gamma_{p_s(x,v)},h}(t))
$$
we have
\begin{align*}
  \frac{d}{ds} d_{\gamma_{p_s(x,v)}(t)}f(J_{\gamma_{p_s(x,v)},h}(t))
  &= \Hess f (\gamma_{p_s(x,v)})(J_p(s,t),J_{\gamma_{p_s(x,v)},h}(t)) \\
  &+ \br{\grad f(\gamma_{p_s(x,v)})(t),D_s J_{\gamma_{p_s(x,v)},h}(t)}.
\end{align*}
Since $\Hess f$ and $\grad f$ are compactly supported
we can move derivative $\frac{d}{ds}$ into integral
and deduce that $VX_\perp u^f (x,v)$ exists for all $(x,v) \in SM$.

Proofs for $V^2 u^f$ and $X_\perp^2 u^f$ are once again similar.
\end{proof}

As a last application of Lemmas
\ref{jacobi growth exponential} and \ref{jacobi growth polynomial} we derive
an estimate for the volumes of spheres in our setting.

\begin{Lemma}\label{volume of sphere}
Suppose $|K| \leq K_0$ and $p \in M$. Then
$$
  \Vol \sphere{p}{r} \leq C(K_0)\e^{\sqrt{K_0}r}.
$$
If $K \in \classP{\eta}(p,M)$ for some $\eta > 2$, then
$$
  \Vol \sphere{p}{r} \leq C t.
$$
\end{Lemma}
\begin{proof}
We use polar coordinates centered at point $p$.
Fix a tangent vector $v \in S_pM$ and define mapping
$f \colon [0,\infty) \times (0,2\pi) \to M$ by
$f(r,\theta) = \exp_p(r \rot{v}{\theta})$.
This gives the usual polar coordinates
in which the metric $g$ takes form
$$
  g(r,\theta) = dr^2 + \left| \frac{df}{d\theta} \right|^2 d\theta^2
$$
and the corresponding volume form is
$$
  dV_g(r,\theta) = \left| \frac{df}{d\theta} \right| dr\wedge d\theta.
$$

Since $\exp_p(r \rot{v}{\theta}) = \gamma_{p_\theta(p,v)}(r)$
we have
\begin{align*}
  \frac{df}{d\theta}(r,\theta)
  = \frac{d}{dt} \gamma_{p_\theta(p,v)}(r)= J_p(r,\theta)
\end{align*}
and hence the volume form on $S_p(r)$ is given by
$$
  \iota_{\doo_r}dV_g(r,\theta) = \left| \frac{df}{d\theta} \right| d\theta
  = J_p(r,\theta) d\theta. 
$$
By Lemma \ref{jacobi growth exponential}
$$
  \Vol \sphere{p}{r} \leq \int_0^{2\pi} C(K_0)\e^{\sqrt{K_0}r} \, d\theta
  = C(K_0)\e^{\sqrt{K_0}r}.
$$

In the presence of the additional assumption for the Gaussian curvauture
Lemma \ref{jacobi growth polynomial} yields
\[
  \Vol \sphere{p}{r} \leq C t.\qedhere
\]
\end{proof}

\section{Pestov identity and $C^2$-approximation}
In this section we prove our main theorems. The proofs are based on
a certain kind of energy estimate for the operator $P = VX$ called the
Pestov identity. We will use Pestov identity with boundary terms on
submanifolds of $(M,g)$.
Througout this section we denote $M_{p,r} = \ball{p}{r} \subset M$, a submanifold of $M$ with
boundary $\sphere{p}{r}$.

The following form of Pestov identity constitutes
the main argument for our proofs of the main theorems.

\begin{Lemma}[\cite{IS16}]\label{pestov}
For $u \in C^2(SM)$ it holds
\begin{align*}
  \norm{VXu}_{L^2(SM_{p,r})}^2
  &= \norm{XVu}_{L^2(SM_{p,r})}^2 + \norm{Xu}_{L^2(SM_{p,r})}^2 - \br{KVu,Vu}_{SM_{p,r}}\\
  &- \br{\br{v,\nu}Vu,X_\perp u}_{\partial SM_{p,r}}
   + \br{\br{v_\perp,\nu}Vu,Xu}_{\partial SM_{p,r}}
\end{align*}
\end{Lemma}

By using approximating sequences we can relax the regularity
assumptions for the Pestov identity.
Especially the Pestov identity holds for $u^f$ with
suitable $f$.

\begin{Lemma}\label{pestov for uf}
Suppose either one of the following:
\begin{enumerate}
  \item $|K(x)| \leq K_0$ and $f \in \classEd{\eta}(p,M)\cap C^2(M)$ for some $\eta > \sqrt{K_0}$.
  \item $f \in \classPd{\eta}(p,M) \cap C^2(M)$ for some $\eta > 1$
  and $K \in \classP{\tilde\eta}(p,M)$ for some $\tilde\eta > 2$.
\end{enumerate} 
If $If = 0$, then  the Pestov identity in Lemma \ref{pestov} holds for $u^f$.
\end{Lemma}
\begin{proof}
Lemmas \ref{Xu = -f} and \ref{derivative estimates}
ensure that all terms of the Pestov identity are finite.

We define $u_k = u^{\varphi_k f}$
where $\varphi_k \colon M \to \mR$ is a smooth
cutoff function such that
\begin{enumerate}
  \item $0 \leq \varphi_r(x) \leq 1$ for all $x \in M$.
  \item $\varphi_k(x) = 1$ for $x \in \ball{p}{k}$.
  \item $\varphi_k(x) = 0$ for $x \not\in \ball{p}{2k}$.
  \item $|\grad \varphi|_g \leq C/k$ for all $x \in M$
  and $v \in T_xM$.
\end{enumerate}
Such a function can be defined by
$$
  \varphi_k(x) \mdef  \varphi\left(\frac{d_g(x,p)}{k}\right)
$$
where $\varphi$ is a suitable smooth cutoff function on $\mR$.
Since functions $\varphi_k$ are smooth and compactly supported,
we have $u_k \in C^2(SM)$ by Lemma \ref{C^2 for compactly supported}.

Let us move on to prove the convergence.
First we observe that
$$
  Xu_k(x,v)\big|_{SM_{p,r}} = -f(x)
$$
for large $k$. Therefore
we have convergence in $L^2$-norm
for the term $Xu_k$.

Next we prove convergence for $XV u_k$
under the assumption that
$f \in \classPd{\eta}(p,M) \cap C^2(M)$ for some $\eta > 1$
and $K \in \classP{\tilde\eta}(p,M)$ for some $\tilde\eta > 2$.
First we notice that
$$
  XVu_k = VXu_k + X_\perp u_k = X_\perp u_k
$$
for large $k$. Similarly $XVu^f = X_\perp u^f$
so it is enough to prove that $X_\perp u_k$ converges
to $X_\perp u^f$.
Furthermore since $SM_{p,r}$ has finite volume it is enough to prove that
$X_\perp u_k \to X_\perp u^f$ in $L^\infty$-norm.

Let us denote $G = \{ \gamma_{x,v} \colon (x,v) \in SM_{p,r} \}$.
The set $G$ fulfills the assumption of Lemma \ref{jaboci on set G}.
Suppose $(x,v) \in SM_r$.
We have
\begin{align*}
  X_\perp u_k(x,v) - X_\perp u^f(x,v)
  &= \int_0^\infty d_{\gamma_{x,v}(t)}(\varphi_k f)(J_h(t)) \, dt \\
  &- \int_0^\infty d_{\gamma_{x,v}(t)}f (J_h(t)) \, dt \\
  &= \int_0^\infty (\varphi_k(\gamma_{x,v}(t))-1)d_{\gamma_{x,v}(t)}f(J_h(t)) \, dt \\
  &+ \int_0^\infty f(\gamma_{x,v}(t))d_{\gamma_{x,v}(t)}\varphi_k(J_h(t)) \, dt.
\end{align*}

For $t \geq 0$ holds
$$
  d_g(\gamma_{x,v}(t),p) \geq t- d_g(x,p) \geq t - r.
$$
Also
$$
  (1 - \varphi_k(\gamma_{x,v}(t)) = 0
$$
at least for $0 \leq t \leq k - r$
and $d_{\gamma_{x,v}(t)}\varphi_k$ can be non-zero only
in interval $[k-r,2k+r]$, which can be seen using
triangle inequality.

Hence we can estimate, with help of Lemma \ref{jaboci on set G}, that
\begin{align*}
  |X_\perp u_k(x,v) - X_\perp u^f(x,v)|
  &\leq \int_{k-r}^\infty |d_{\gamma_{x,v}(t)}f(J_h(t))| \, dt \\
  &+ \int_{k-r}^{2k+r} |f(\gamma_{x,v}(t))d_{\gamma_{x,v}(t)}\varphi_k(J_h(t))| \, dt\\
  &\leq C_1\int_{k-r}^\infty \frac{t}{(1 + d(\gamma_{x,v}(t),p))^{\eta+1}} \, dt \\
  &+ \frac{C_2}{k}\int_{k-r}^{2k+r} \frac{t}{(1 + d(\gamma_{x,v}(t),p))^{\eta+1}} \, dt\\
  &\leq C_1\int_{k-r}^\infty \frac{t}{(1 + t-r)^{\eta+1}} \, dt \\
  &+ \frac{C_2}{k}\int_{k-r}^{2k+r} \frac{t}{(1 + t-r)^{\eta+1}} \, dt.
\end{align*}
The last two integrals do not depend on $(x,v)$ and
they also tend to zero as $k \to \infty$,
which proves the $L^\infty$-convergence.
In similar manner we can prove
convergence for $Vu_k$.

Convergence for the boundary terms follows also
from the $L^\infty$-convergence because
the boundary $\partial SM_{p,r}$ has a finite volume.

In the other case we proceed similarly but use
Lemma \ref{jacobi growth exponential} instead of
Lemma \ref{jaboci on set G}.
\end{proof}

We are ready to prove our main theorems.
\begin{proof}[Proof of Theorem 1]
Since the geodesic ray transform is linear
it is enough to show that $If = 0$
implies $f = 0$.

Let us assume $f \in \classEd{\eta}(p,M) \cap C^2(M)$,
$\eta > \frac{5}{2}\sqrt{K_0}$,
is such that $If = 0$.
Lemma \ref{pestov for uf} tell us that Pestov identity
holds for $u^f$. We will
apply it on submanifold $SM_{p,r}$.

Since $Xu^f = -f$, the term on the left hand side
of the Pestov identity is zero.
Because we assume Gaussian curvature to be non-positive we
have
$$
  -\br{KVu^f,Vu^f}_{SM_{p,r}} \geq 0.
$$
Thus if we can show that the two boundary terms
tend to zero as $r \to \infty$, it must be
that 
$$
  \lim_{r \to 0}\norm{Xu^f}_{L^2(SM_{p,r})}
  = \lim_{r \to 0}\norm{f}_{L^2(SM_{p,r})}
  = 0
$$
which proves the injectivity.

Using Lemma \ref{derivative estimates}
together with Lemma \ref{volume of sphere}
gives
\begin{align*}
%  \left| \int_{\partial SM_{p,r}} \br{v,\nu}(Vu^f)(X_\perp u^f) \, d\Sigma^2 \right|
  \left|\br{\br{v,\nu}Vu,X_\perp u}_{\partial SM_{p,r}}\right|
  &\leq \int_{\partial SM_{p,r}} |Vu^f||X_\perp u^f| \, d\Sigma^2 \\
  &\leq C(\eta,K_0) \int_{\partial M_{p,r}} \int_{S_xM} \e^{2(2\sqrt{K_0}-\eta)d_g(x,p)} \, dS \, dV_g \\
  &\leq C(\eta,K_0) \int_{\partial M_{p,r}} \e^{2(2\sqrt{K_0}-\eta)r} \, dV_g \\
  &\leq C(\eta,K_0) \int_{\partial M_{p,r}} \e^{2(2\sqrt{K_0}-\eta)r} \, dV_g \\
  &\leq C(\eta,K_0) \e^{2(2\sqrt{K_0}-\eta)r}\Vol \sphere{p}{r} \\
  &\leq C(\eta,K_0) \e^{(5\sqrt{K_0} - 2\eta)r},
\end{align*}
which indeed tends to zero as $r \to \infty$.

Similarly we obtain
\begin{align*}
  \left| \int_{\partial SM_{p,r}} \br{v_\perp,\nu} (Vu^f)(X u^f) d\Sigma^2 \right|
  &\leq C(\eta,K_0) \e^{(3\sqrt{K_0} - 2\eta)r}.
\end{align*}
which also tends to zero as $r \to \infty$.
\end{proof}

\begin{proof}[Proof of Theorem 2]
The proof is as for the Theorem 1, just using
the other estimates provided by Lemmas
\ref{derivative estimates} and \ref{volume of sphere}.
\end{proof}

\bibliography{refs}{}
\bibliographystyle{alpha}

\end{document}